\documentclass[10pt]{article}
\usepackage{geometry}                
\usepackage[parfill]{parskip}   
\usepackage{graphicx}
\usepackage{amssymb}
\usepackage{amsfonts}
\usepackage{amsmath} 
\usepackage{amsthm}
\usepackage{empheq}
\usepackage{esint}
\usepackage{epstopdf}
\usepackage{verbatim}
\usepackage{ifthen}
\usepackage{comment}
\usepackage{hyperref}

\begin{document}

\title{Well-posedness and asymptotic behavior\\ of a multidimensional model of morphogen transport.}

\author{Marcin Ma\l ogrosz\footnote{Institute of Applied Mathematics and Mechanics, University of Warsaw, Banacha 2, 02-097 Warsaw, Poland (malogrosz@mimuw.edu.pl)}
}

\date{for Kate}

\newtheorem{theo}{Theorem}
\newtheorem{lem}{Lemma}
\newtheorem*{rem}{Remark}

\newcommand{\n}[1]{\lVert#1\rVert}
\newcommand{\eq}[1]{\begin{align}#1\end{align}}
\providecommand{\bs}{\begin{subequations}}
\providecommand{\es}{\end{subequations}}
\newcommand{\com}[1]{}

\maketitle

\begin{abstract}
\noindent
Morphogen transport is a biological process, occurring in the tissue of living organisms, which is a determining step in cell differentiation. We present rigorous analysis of a simple model of this process, which is a system coupling parabolic PDE with ODE. We prove existence and uniqueness of solutions for both stationary and evolution problems. Moreover we show that the solution converges exponentially to the equilibrium in $C^{1,\alpha}\times C^{0,\alpha}$ topology. We prove all results for arbitrary dimension of the domain. Our results improve significantly previously known results for the same model in the case of one dimensional domain.

\end{abstract}



\section{Introduction}

Morphogen transport (MT) is a biological process occurring in the bodies of living organisms. It is known that certain proteins (ligands) act as the morphogen - a conceptually defined substance which is responsible for the development of the shape, size and other properties of the cells. According to the 'French flag model' of Wolpert \cite{Wol}, morphogen molecules spread from a localized source through the tissue of newly born individuals and after some time form stable gradients of concentrations. Receptors, located on the surface of the cells, detect those gradients and pass to the kernels the information about levels of morphogen concentration. Then according to these information, certain mechanisms begin synthesis of proteins which finally results in cell differentiation and specialization. Although the role of morphogen gradient in gene expression seems to be widely accepted, the exact kinetic mechanism of its formation is still not known. (see \cite{GB},\cite{KW} and \cite{KPBKBJG-G}).

Recently various models consisting of PDE-ODE systems were proposed to explain MT. Those models assume that movement of morphogen molecules occurs by different types of diffusion or by chemotaxis in the extracellular medium. Reactions with receptors (reversible binding, transcytosis) and various possibilities of degradation and internalization (of morphogens, receptors, morphogen-receptor complexes) are also being considered (see \cite{LNW}, \cite{KLW2}, \cite{BKPG-GJ}, \cite{STW}).

For the case of morphogen Decapentaplegic (Dpp) acting in the wing disc of the Drosophila Melanogaster individuals,  several models have been proposed in \cite{LNW}. 
In this paper we will be concerned with model \textbf{[LNW].B} (Model B \cite{LNW} p786). In this model it is assumed that movement of morphogen molecules occurs by passive diffusion while being affected by reactions of reversible binding with receptors and degradation of morphogen-receptor complexes.
Morphogen is being delivered to the system by secretion from a source localized on one of the boundaries of the domain $\Omega\subset\mathbb{R}^n$, which represents a fragment of the wing tissue. In mathematical terms the model is a system of two differential equations (PDE+ODE equipped with initial and boundary conditions), governing time evolution of the concentrations of free morphogen and morphogen-receptor complexes. 

In case of 1D domains a detailed mathematical analysis of this model was made in \cite{Tel} and \cite{KLW1}. 

In \cite{Tel} the case $\Omega=(0,\infty)$, with a nonlinear dynamic boundary condition at $x=0$ and vanishing boundary condition at $x\to\infty$ is considered. Well-posedness and $L_p(\Omega)$ convergence of the solution to unique steady state were proved.

In \cite{KLW1} the case $\Omega=(0,1)$, with nonhomogeneous, constant Neumann condition at $x=0$ and homogeneous Dirichlet condition at $x=1$ is analyzed. Finding Lyapunov functional allowed to prove well-posedness and $L_2(\Omega)$ exponential convergence to the unique equilibrium, with rate $\chi$ expressed explicitly by the parameters of the model.

The goal of this paper is to examine \textbf{[LNW].B} in the \cite{KLW1} setting for bounded domains of arbitrary dimension $n$. Although $n\in\{1,2,3\}$ is, from the biological point of view, the only relevant case, we do not put this restriction on $n$ (methods that we use do not depend on the dimension). Using fixed point theorem and monotonicity of the nonlinearity we prove that our model has a unique nonnegative steady state. Using theory of analytic semigroups and comparison principle arguments we show existence of classical global solutions. 
We check that Lyapunov functional, obtained in \cite{KLW1}, also works for arbitrary $n$ and thanks to appropriate semigroup estimates and bootstrap arguments we improve the topology of the convergence to the equilibrium from $L_2\times L_2$ to $C^{1,\alpha}\times C^{0,\alpha}$ without losing the exponential rate $\chi$. 

\section{The model}

We consider the system of differential equations governing the space and time evolution of the concentration of free morphogen $l$ and concentration of bounded receptors $s$ in an annular shape domain $\Omega\subset\mathbb{R}^n$. We assume that receptors are distributed uniformly in the tissue so after normalizing the total concentration of receptors (free+bounded) is equal to 1. The model governs the following biological processes
\begin{itemize}
\item Passive diffusion of morphogens in the extracellular medium.
\item Secretion of morphogens from the source on a subset $\Gamma_{N}$ of $\partial\Omega$.
\item Binding of morphogens to receptors.
\item Unbinding of morphogens from receptors.
\item Degradation of bounded morphogens.
\end{itemize}
 
We equip the model with initial conditions $l_{0}, s_{0}$ and boundary conditions on $\Gamma_{D}, \Gamma_{N}$ - two disjoint parts of $\partial\Omega$. On $\Gamma_{N}$ we consider nonhomogeneous, time independent, nonnegative Neumann condition (flow of morphogen into the domain) while on $\Gamma_{D}$ we put homogeneous Dirichlet condition (far from the source of morphogen their impact on the whole process is negligible).
After normalization we end up with the following model
\newline
\\
\textbf{[LNW].B}
\begin{align*}
\partial_t l-D\Delta l&=\delta s-l(1-s) && ,(t,x)\in(0,\infty)\times\Omega\\
\partial_t s&=-(\delta+\epsilon)s+l(1-s) && ,(t,x)\in(0,\infty)\times\Omega\\
-D\nabla_{n}l&=-\nu && ,(t,x)\in(0,\infty)\times\Gamma_{N}\\
l&=0 && ,(t,x)\in(0,\infty)\times\Gamma_{D}\\
l(0)&=l_{0}&& ,x\in\Omega\\
s(0)&=s_{0}&& ,x\in\Omega
\end{align*}

where we denote the derivative in the direction of the outer normal vector to $\Gamma_{N}$ by $\nabla_{n}$.

\section{Results}

In the whole paper we assume that
\begin{description}\label{A}
\item[\textbf{A1}] $n\in\mathbb{N},\ p>n\geq1$.
\item[\textbf{A2}] $\Omega\subset\mathbb{R}^n$ is a 
bounded domain (open, connected) with ($\mathcal{C}^{1,1}$) boundary which consists of two disjoint parts: $\partial\Omega=
\Gamma_D\sqcup\Gamma_N$.
\item[\textbf{A3}] $0\leq\nu\in W_p^{1-1/p}(\Gamma_N)$.
\item[\textbf{A4}] $l_0,s_0\in W^1_{p}(\Omega);\quad 0\leq l_{0}(x),\quad0\leq s_{0}(x)<1$, for $x\in\Omega;\quad l_{0}(x)=s_{0}(x)=0$, for $x\in\Gamma_{D}
$.
\end{description}

Under the above assumptions we first analyze the stationary problem and prove the following
\
\newline

\begin{theo}
\label{equilibrium}
\textbf{[LNW].B} has unique nonnegative steady state $(l_{\infty},s_{\infty})$, where \\ $0\leq l_{\infty}\in W^{2}_{p}(\Omega)$ is the unique solution to
\bs{\label{steady}}
\eq{
-D\Delta l_{\infty}&=-\frac{\epsilon l_{\infty}}{\delta+\epsilon+l_{\infty}}&&, x\in\Omega\label{steady:a}\\ 
-D\nabla_n l_{\infty}&=-\nu&&,x\in\Gamma_{N}\label{steady:b}\\
l_{\infty}&=0&&,x\in\Gamma_{D}\label{steady:c}.
}
\es
and $s_{\infty}=l_{\infty}/(\epsilon+\delta+l_{\infty})$.
\end{theo}

The proof of existence is based on maximal regularity for uniformly elliptic operators in Sobolev spaces, compact embedding, comparison principle and Schauder fixed point theorem. Uniqueness follows from monotonicity of the nonlinear part in \eqref{steady:a}.	
\
\newline

We next turn to the evolution problem and establish its well-posedness. 
\
\newline

\begin{theo}
\label{global}
\textbf{[LNW].B} has unique solution $(l,s)$ such that
\bs\label{solution}
\eq{
l-l_{\infty}&\in\mathcal{C}([0,\infty);W^{1}_p(\Omega))\cap\mathcal{C}^{1}((0,\infty);W^{1}_p(\Omega))\cap\mathcal{C}((0,\infty);W^{3}_{p}(\Omega))\label{solution:l}\\
s&\in\mathcal{C}^1([0,\infty);W^{1}_p(\Omega))\label{solution:s}.
}
Moreover for $(t,x)\in[0,\infty)\times\Omega$
\eq{
0\leq l(t,x),\quad 0\leq s(t,x)<1\label{Sign}.
}
\es
\end{theo}
\
\newline
Local existence and uniqueness are obtained by putting system \textbf{[LNW].B} into the semigroup framework and using general theory for abstract parabolic semilinear problems. Comparison principle allows us to deduce that \eqref{Sign} is satisfied from which we get that our solution is global. 
\
\newline

We finally study the stability of the steady state and show that it attracts all trajectories with the uniform exponential rate.
\
\newline
\begin{theo}
\label{asymptotics}
There exists a positive constant $C$ depending on $l_0,s_0,\nu,\delta,\epsilon,D,\Omega,p$ such that for every $t>0$
\bs
\eq{
\n{l(t)-l_{\infty}}_{1,p}+\n{s(t)-s_{\infty}}_{1,p}&\leq Ce^{-(\chi/2)t}\label{asym1},\\
\n{l(t)-l_{\infty}}_{2,p}&\leq C\max\{1/\sqrt{t},1\}e^{-(\chi/2)t}\label{asym2},
}
where
\eq{
\chi=\min\Big\{D\lambda_{1},\frac{D\lambda_1(\delta+\epsilon)}{2(D\lambda_1+2)}+\frac{\epsilon}{2}\Big\}\label{chi}
}
\es
and $\lambda_{1}$ is defined in Lemma \ref{semigroup}.
\end{theo}

By extending Lyapunov functional (derived in \cite{KLW1} for one dimensional interval) to the case of arbitrary dimension we obtain estimates on the distance between solution and steady state in $L_{2}\times L_{2}$ topology. Using regularising properties of the heat semigroup we next bootstrap the topology of convergence to $W^2_p\times W^1_p$.
\\
\newline
\textbf{Remark}\\
Using embedding $W^2_p(\Omega)\times W^1_p(\Omega)\subset C^{1,\alpha}(\Omega)\times C^{0,\alpha}(\Omega)$ valid for $p>n, \ 0\leq\alpha\leq1-n/p$ we obtain topology of convergence as claimed in the introduction.


\section{Notation, semigroup estimates, Gronwall inequality}

For $x,y\in\mathbb{R}$ we denote $x\vee y:=\max\{x,y\},\, x\wedge y:=\min\{x,y\},\,x_{+}:=x\vee 0,\, x_{-}:=(-x)\vee 0$ and extend this notion to real valued functions. If $(V,\geq)$ is partially ordered vector space we denote its positive cone by $V_{+}:=\{v\in V\colon \ v\geq 0\}$. 

We make standard convention that $C$ denotes positive constant which may depend on a subset of $\{l_0,s_0,\nu,\delta,\epsilon,D,\Omega,p\}$ and may change its value from line to line.

For $1<q<\infty, \alpha\in\{1,2,3\}$ we introduce the spaces $W^{\alpha}_{q,\mathcal{B}^{\alpha}}(\Omega)$:
\begin{align*}
W^{1}_{q,\mathcal{B}^1}(\Omega)&=\{u\in W^{1}_{q}(\Omega): u|_{\Gamma_{D}}=0\} 
\\
W^{2}_{q,\mathcal{B}^2}(\Omega)&=\{u\in W^{2}_{q}(\Omega): u|_{\Gamma_{D}}=0,\, \nabla_{n}u|_{\Gamma_{N}}=0 \} \\
W^{3}_{q,\mathcal{B}^3}(\Omega)&=\{u\in W^{3}_{q}(\Omega): u|_{\Gamma_{D}}=0,\, \nabla_{n}u|_{\Gamma_{N}}=0,\, \Delta u|_{\Gamma_{D}}=0\}, 
\end{align*}

with standard Sobolev norms $\n{.}_{\alpha,q}$. 
\newline

We next recall some properties of the heat semigroup generated by laplacian with appropriate boundary conditions. 
\
\newline

\begin{lem}\label{semigroup}
For $1<q<\infty$ the Laplace operator $\Delta\colon L_q(\Omega)\supset W^{2}_{q,\mathcal{B}^2}(\Omega)\to L_q(\Omega)$ generates an analytic, strongly continuous semigroup $e^{t\Delta}$. For $\alpha,\beta\in\{0,1,2,3\},\ \alpha\leq\beta, \ 1<q_1\leq q_2<\infty$ and $t>0$ we have
\bs\label{estimates}
\eq{
\n{e^{t\Delta}u}_{\beta,q}&\leq C(t\wedge 1)^{(\alpha-\beta)/2}e^{-\lambda_{1}t}\n{u}_{\alpha,q}\leq Ct^{(\alpha-\beta)/2}\n{u}_{\alpha,q},&& u\in W^{\alpha}_{q,\mathcal{B}^{\alpha}}\label{estimate:1}\\
\n{e^{t\Delta}u}_{q_2}&\leq C(t\wedge 1)^{-n/2(1/{q_1}-1/{q_2})}e^{-\lambda_{1}t}\n{u}_{q_1}\leq Ct^{-n/2(1/{q_1}-1/{q_2})}\n{u}_{q_1},&& u\in L_{q_1}\label{estimate:2}
}

\es
where $\lambda_{1}>0$ is the first eigenvalue of $-\Delta$ and $C$ depends only on $q,q_1,q_2,\Omega$.
\end{lem}

\begin{proof}
Noticing that $-\lambda_1=\sup Re(\sigma(\Delta))$ we get from \cite{Lun} following estimates
\begin{align*}
\n{e^{t\Delta}u}_q&\leq M_0e^{-\lambda_1 t}\n{u}_q\\
\n{t(\Delta+\lambda_1I)e^{t\Delta}u}_q&\leq M_1e^{-\lambda_1 t}\n{u}_q
\end{align*}
We have
\begin{align*}
\n{e^{t\Delta}u}_{2,q}&\leq C\n{\Delta e^{t\Delta}u}_q\leq C\n{(\Delta+\lambda_1I)e^{t\Delta}u}_q+C\lambda_1\n{e^{t\Delta}u}_q\\
&\leq C(M_1/t+M_{0}\lambda_1)e^{-\lambda_1t}\n{u}_q\leq C(t\wedge 1)^{-1}e^{-\lambda_1t}\n{u}_q
\end{align*}
From \cite{Ama} we have 
that 
\begin{align*}
[L_q,W^2_{q,B^2}]_{\alpha/2}&=W^{\alpha}_{q,B^{\alpha}}, \ \alpha\in\{0,1,2,3\},\\
[L_{q_1},W_{q_1,B^2}^2]_{\theta}&\subset L_{q_2}, \ \theta\geq n/2(1/q_1-1/q_2),
\end{align*}
where for $\theta\in[0,1] \ [.,.]_{\theta}$ denotes complex interpolation functor, which is extended for $\theta>1$ as described in \cite{Ama}. From this estimates \eqref{estimate:1} and \eqref{estimate:2} follows.
\end{proof}
\
\newline
We next recall the singular Gronwall inequality
\
\newline

\begin{lem}
\label{Gronwall}
Assume that $f\in\mathcal{C}([0,T);\mathbb{R}^{+})$ satisfies for every $t\in[0,T)$ following inequality
$$
f(t)\leq a+b\int_{0}^{t}(t-s)^{-\alpha}f(s)ds,
$$
where $a,b$ are nonnegative constants and $\alpha\in[0,1)$.
Then there exists positive constant $C=C(b,\alpha)$ such that for $t\in[0,T)$
$$
u(t)\leq aC e^{bCt}.
$$
Moreover $C(b,0)=1$.
\end{lem}
\begin{proof}
For proof (under more general assumptions) see Lemma 7.1.1 in \cite{Hen}.
\end{proof}

\section{Proof of Theorem \ref{equilibrium}}

For $x\geq 0$ let $f(x)=\frac{\epsilon}{\delta+\epsilon+x}$. 
Consider the operator $T:L_{p}(\Omega)_{+}\to L_{p}(\Omega)$, defined by $T(u)=w$ where $w\in W^{2}_{p}(\Omega)$ is the unique solution of
\bs{\label{steady'}}
\eq{
-D\Delta w+f(u)w&=0&&, x\in\Omega\label{steady':a}\\ 
-D\nabla_n w&=-\nu&&,x\in\Gamma_{N}\label{steady':b}\\
w&=0&&,x\in\Gamma_{D}\label{steady':c}
} 
\es
We will show that $T$ has bounded range in $L_{p}(\Omega)_{+}$, is compact and continuous (this via the Schauder theorem will imply existence of a solution of \eqref{steady} in $W^{2}_{p}(\Omega)$). Using the fact that $0\leq f(x)\leq\frac{\epsilon}{\epsilon+\delta}$ we get from maximal regularity of uniformly elliptic differential operators in Sobolev spaces (see \cite{Gri} for instance) the following estimate
$$
\n{w}_{W^{2}_{p}(\Omega)}\leq C\n{\nu}_{W^{1-1/p}_{p}(\Gamma_{N})}
$$
which gives boundedness of the range of $T$ in $W^{2}_{p}(\Omega)$ and therefore in $L_{p}(\Omega)$. Compactness of $T$ follows from the compact imbedding $W^{2}_{p}(\Omega)\subset\subset L_{p}(\Omega)$. To show that $w\geq 0$ we multiply \eqref{steady':a} by $w_{-}$ and integrate by parts (notice that for $p>n$ $w\in W^2_p(\Omega)\subset W^1_2(\Omega)$ hence $w_-\in W^1_2(\Omega)$) to obtain
$$
-D\int_{\Omega}|\nabla w_-|^2-\int_{\Gamma_{N}}\nu w_{-}-\int_{\Omega}f(u)w_{-}^2=0.
$$
Since $w=0$ on $\Gamma_{D}$ therefore $w\geq0$ in $\Omega$.\\
Assume that $u_{n}\to u$ in $L_{p}(\Omega)$. Let $w=T(u), w_{n}=T(u_{n})$, then 
\begin{align*}
-D\Delta (w_{n}-w)+f(u_{n})(w_{n}-w)+w(f(u_{n})-f(u))&=0&&,x\in\Omega\\
-D\nabla_{n}(w_{n}-w)&=0&&, x\in\Gamma_{N}\\
w_{n}-w&=0&&, x\in\Gamma_{D}
\end{align*}
therefore
$$
\n{w_{n}-w}_{L_{p}(\Omega)}\leq C\n{w(f(u_{n})-f(u))}_{L_{p}(\Omega)}\leq C\n{w}_{L_{\infty}(\Omega)}\n{f'}_{L_{\infty}(0,\infty)}\n{u_n-u}_{L_p(\Omega)}$$ 
which proves that $T$ is continuous. 
Using Schauder fixed point theorem we obtain existence of $l_{\infty}\in W^{2}_{p}(\Omega)$ which solves \eqref{steady}. \\
To prove uniqueness, assume that $l^{1}_{\infty},l^{2}_{\infty}$ are solutions of \eqref{steady}. Subtracting equations \eqref{steady:a} for $l^{1}_{\infty},l^{2}_{\infty}$, multiplying by $l^{1}_{\infty}-l^{2}_{\infty}$, integrating by parts and using the monotonicity of function $\mathbb{R}_{+}\ni x\to xf(x)$ we get
$$
-D\int_{\Omega}|\nabla (l^{1}_{\infty}-l^{2}_{\infty})|^2  = \int_{\Omega}(f(l^{1}_{\infty})l^{1}_{\infty}-f(l^{2}_{\infty})l^{2}_{\infty})(l^{1}_{\infty}-l^{2}_{\infty})\geq0,
$$
which by \eqref{steady:c} implies $l^{1}_{\infty}\equiv l^{2}_{\infty}$.

\section{Proof of Theorem \ref{global}}

To deal with nonhomogeneous boundary condition on $\Gamma_{N}$ we subtract from $(l,s)$ the stationary state $ (l_{\infty},s_{\infty})$. Setting $(z_1,z_2)=(l-l_{\infty},s-s_{\infty})$ we arrive at
\bs\label{modelz}
\eq{
\partial_{t}z_1-D\Delta z_1&=\delta z_2-z_1(1-z_2)+s_{\infty}z_1+l_{\infty}z_2 && ,(t,x)\in(0,\infty)\times\Omega\label{modelz:a}\\
\partial_{t}z_2&=-(\delta+\epsilon)z_2+z_1(1-z_2) -s_{\infty}z_1-l_{\infty}z_2&& ,(t,x)\in(0,\infty)\times\Omega\label{modelz:b}\\
-D\nabla_{n}z_1&=0 && ,(t,x)\in(0,\infty)\times\Gamma_{N}\\
z_1&=0 && ,(t,x)\in(0,\infty)\times\Gamma_{D}\\
z_1(0)&=z_{10}=l_{0}-l_{\infty}&& ,x\in\Omega\\
z_2(0)&=z_{20}=s_{0}-s_{\infty}&& ,x\in\Omega
} 
\es
\noindent
We interpret system \eqref{modelz} as a differential equation in a Banach space specified below
\bs\label{ODE}
\eq{
\dot{z}-\mathcal{A}z&=H(z)&&,t\in(0,\infty)\label{ODE:a}\\
z(0)&=z_{0}=(z_{10},z_{20})&&\label{ODE:a} 
}
\es
where $z=(z_1,z_2),\ \mathcal{A}z=(D\Delta z_1,0),\ H=(H^1,H^2),$
\bs\label{H}
\eq{
H^1(z)&=\delta z_2-z_1(1-z_2)+s_{\infty}z_1+l_{\infty}z_2\\
H^2(z)&=-(\delta+\epsilon)z_2+z_1(1-z_2) -s_{\infty}z_1-l_{\infty}z_2.
}
\es

In the following lemma we prove local existence for \eqref{ODE}.
$\\$

\begin{lem}\label{local}
For $\alpha\in\{0,1,2,3\}$ denote $Z_{\alpha,p}=W^{\alpha}_{p,\mathcal{B}^{\alpha}}\times W^{1}_{p,\mathcal{B}^1}$. For every $z_{0}\in Z_{1,p}$ the Cauchy problem \eqref{ODE} possess a unique maximal local solution
\begin{align*}
z\in\mathcal{C}([0,T_{\max});Z_{1,p})\cap\mathcal{C}^1((0,T_{\max});Z_{1,p})\cap\mathcal{C}((0,T_{\max});Z_{3,p}).
\end{align*}
which satisfies for $t\in[0,T_{\max})$ the following Duhamel formula:
\bs\label{Duhamel}
\eq{
z_1(t)&=e^{tD\Delta}z_{10}+\int_0^t e^{(t-s)D\Delta}H^1(z(s))ds\label{Duhamel1}\\
z_2(t)&=z_{20}+\int_0^t H^2(z(s))ds\label{Duhamel2}.
}
\es

Moreover if $T_{\max}<\infty$ then $\limsup_{t\to T_{\max}^{-}}\n{z(t)}_{1,p}=\infty$.
\end{lem}

\begin{proof}

The operator $\mathcal{A}:Z_p\supset Z_{2,p}\to Z_p$ is a generator of an analytic strongly continuous semigroup $e^{t\mathcal{A}}=e^{tD\Delta}\times Id$ (as a product of two generators). Moreover since $Z_{1,p}$ is a Banach algebra ($p>n$) we observe that $H:Z_{1,p}\to Z_{1,p}$ is locally Lipschitz on bounded sets. The claim follows from Theorem 7.2.1 in \cite{ChD}.
\end{proof}

We next turn to the proof of \eqref{Sign}.

To prove that for $t\in[0,T_{\max})$ $l(t),s(t)\geq 0$ we consider the system
\bs\label{model+}
\eq{
\partial_{t}l'-D\Delta l'&=\delta s'_{+}-l'_{+}(1-s'_{+})&& ,(t,x)\in(0,\infty)\times\Omega\label{model+:a}\\
\partial_{t}s'&=-(\delta+\epsilon)s'_{+}+l'_{+}(1-s'_{+})&& ,(t,x)\in(0,\infty)\times\Omega\label{model+:b}\\
-D\nabla_{n}l'&=-\nu&& ,(t,x)\in(0,\infty)\times\Gamma_{N}\label{model+:c}\\
l'&=0&& ,(t,x)\in(0,\infty)\times\Gamma_{D}\label{model+:d}\\
l'(0)&=l_{0}&& ,x\in\Omega\label{model+:e}\\
s'(0)&=s_{0}&&, x\in\Omega\label{model+:f}
}
\es
As before one can show that \eqref{model+} possess unique classical local solution $(l',s')$. After multiplying \eqref{model+:a} by $l_{-}$ and integrating by parts we obtain
\begin{align*}
-\frac{1}{2}\frac{d}{dt}\int_{\Omega}|l'_{-}|^2dx-D\int_{\Omega}|\nabla l'_{-}|^2dx-D\int_{\Gamma_{N}}l'_{-}\nu dS=\delta\int_{\Omega}s'_{+}l'_{-}dx\geq0.
\end{align*}
Similarly multiplying \eqref{model+:b} by $s_{-}$ yields
\begin{align*}
-\frac{1}{2}\frac{d}{dt}\int_{\Omega}|s'_{-}|^2dx=\int_{\Omega}l'_{+}s_{-}'dx\geq0.
\end{align*}

Therefore for $t\in[0,T_{\max})$
\begin{align*}
\n{l'(t)_{-}}_{2}^2+\n{s'(t)_{-}}_{2}^2\leq\n{l'(0)_{-}}_{2}^2+\n{s'(0)_{-}}_{2}^2=0
\end{align*}
and consequently $l'(t)\geq0,s'(t)\geq0$. We observe now that $(l',s')$ is a solution of \textbf{[LNW].B} and using uniqueness we finally get that $l(t)=l'(t)\geq0, s(t)=s'(t)\geq0$ for $t\in[0,T_{\max})$.

To show that $s(t,x)<1$ for $(t,x)\in[0,T_{\max})\times\overline{\Omega}$ we get from Lemma \ref{local}, that for every fixed $x\in\overline{\Omega}$ the function $\underline{s}=1-s=1-z_2-s_\infty\in\mathcal{C}^{1}([0;T_{\max}),\mathbb{R})$ satisfies for $t>0$ the following ODE
\begin{align*}
\underline{\dot{s}}+(\delta+\epsilon+l)\underline{s}&=\delta+\epsilon.
\end{align*}
Therefore
\begin{align*}
\underline{s}(t)=e^{-(\delta+\epsilon)t-\int_{0}^{t}l(\tau)d\tau}(1-s_{0})+(\delta+\epsilon)\int_{0}^{t}e^{-(\delta+\epsilon)(t-t')-\int_{0}^{t-t'}l(\tau)d\tau}dt'>0.
\end{align*}

We finally show that $T_{\max}=\infty$.
Reasoning by contradiction assume that $T_{\max}<\infty$.
Using uniform $L_{\infty}$ boundedness of $s$ (and therefore of $z_2$) we obtain for $t\in(0,T_{\max})$:
\eq{
\n{H^1(z(t))}_{p}\leq C(1+\n{z_1(t)}_{p})\leq C(1+\n{z_1(t)}_{1,p}).\label{H1_estimate}
}
Using \eqref{Duhamel1},\eqref{estimate:1},\eqref{H1_estimate} we obtain
\begin{align*}
\n{z_1(t)}_{1,p}&\leq\n{e^{tD\Delta}z_{10}}_{1,p}+\int_{0}^{t}\n{e^{(t-\tau)D\Delta}H^1(z(\tau))}_{1,p}d\tau\\
&\leq C\n{z_{10}}_{1,p}+C\int_{0}^{t}(t-\tau)^{-1/2}\n{H^1(z(t))}_{p}d\tau\\
&\leq C\n{z_{10}}_{1,p}+C\int_{0}^{t}(t-\tau)^{-1/2}(1+\n{z_1(\tau)}_{1,p})d\tau\\
&\leq C(\n{z_{10}}_{1,p}+1)+C\int_{0}^{t}(t-\tau)^{-1/2}\n{z_1(\tau)}_{1,p}d\tau
\end{align*}
Using Lemma \ref{Gronwall} we get
that $\n{z_1(t)}_{1,p}\leq C$ and therefore 
\eq{
\n{H^2(z(t))}_{1,p}\leq C(1+\n{z_2(t)}_{1,p})\label{H2_estimate}.
}
Using \eqref{Duhamel2} and \eqref{H2_estimate} we obtain
\begin{align*}
\n{z_2(t)}_{1,p}&\leq\n{z_{20}}_{1,p}+\int_{0}^{t}\n{H^2(z(\tau))}_{1,p}d\tau\leq\n{z_{20}}_{1,p}+C\int_{0}^{t}(1+\n{z_2(\tau)}_{1,p})d\tau\\
&\leq C(\n{z_{20}}_{1,p}+1)+C\int_{0}^{t}\n{z_2(\tau)}_{1,p}d\tau.
\end{align*}
Another application of Lemma \ref{Gronwall} gives desired contradiction from which we deduce that $T_{\max}=\infty$.

\section{Proof of theorem \ref{asymptotics}}
The proof of Theorem \ref{asymptotics} is based on $L_2$ estimates obtained for $n=1$ in \cite{KLW1} and bootstrap method to improve convergence from $X_i$-topology to $X_{i+1}$-topology, where $X_{i+1}\subset X_{i}$ are appropriately chosen Banach spaces. We use (as long as the regularity of our solution permits) the following two step 

\textbf{Bootstrap scheme} 
\begin{enumerate}
\item $\n{z_1(t)}_{X_i}+\n{z_2(t)}_{X_i}\leq Ce^{-(\chi/2)t}$ gives $\n{z_1(t)}_{X_{i+1}}\leq Ce^{-(\chi/2)t}$.
\item $\n{z_1(t)}_{X_{i+1}}\leq Ce^{-(\chi/2)t}$ gives $\n{z_2(t)}_{X_{i+1}}\leq Ce^{-(\chi/2)t}$.
\end{enumerate}

Step 1. is a consequence of Duhamel formula \eqref{Duhamel1} and semigroup estimates \eqref{estimates}. \\Step 2. follows from the fact that we can solve equation \eqref{modelz:b} explicitly for $z_2$ in terms of $z_1$.

\subsection{$L_{2}$ estimate}

We first show that, as in the one dimensional case \textbf{[LNW].B} has a Lyapunov functional from which exponential convergence to the equlibrium  $(l_{\infty},s_{\infty})$ follows.
\
\newline
\begin{lem}
For $x\in[0,1), u,v\in W^1_{p,\mathcal{B}^1}(\Omega), 0\leq v<1$, define
\begin{align*}
\Sigma_{I}(x)&=-\ln(1-x)\\
\Lambda_{0}(v)&=\int_{\Omega}(1-s_{\infty})(l_{\infty}+\delta+2\epsilon)\Big[\Sigma_{I}(v)-\Sigma_{I}(s_{\infty})-\frac{v-s_{\infty}}{1-s_{\infty}}\Big]dx\\
\Lambda(u,v)&=\frac{1}{2}\n{u-l_{\infty}}_{2}^2+\Lambda_{0}(v)\\
\mathcal{D}_{\Lambda}(u,v)&=D\n{\nabla (u-l_{\infty})}_{2}^{2}+\int_{\Omega}\frac{[u(1-v)-(\delta+\epsilon)v]^{2}+\epsilon(l_{\infty}+\delta+\epsilon)(v-s_{\infty})^2}{1-v}dx.
\end{align*}
Then for $t\geq0$
\begin{align*}
\Lambda(l(t),s(t))+\int_{0}^{t}\mathcal{D}_{\Lambda}(l(\tau),s(\tau))d\tau&=\Lambda(l_{0},s_{0})\\
\chi\Lambda(l(t),s(t))&\leq\mathcal{D}_{\Lambda}(l(t),s(t))\\
(\delta+\epsilon)\n{s(t)-s_{\infty}}_{2}^2&\leq2\Lambda_{0}(s(t))
\end{align*}
and
\begin{align}
\n{l(t)-l_{\infty}}_{2}^2+(\delta+\epsilon)\n{s(t)-s_{\infty}}_{2}^2&\leq2\Lambda(l_0,s_0)e^{-\chi t}\label{KLW},
\end{align}
where $\chi$ satisfies \eqref{chi}.

\end{lem}
\begin{proof}
Proof can be obtained exactly as in \cite{KLW1} (part of Theorem 8 and Proposition 9 pp 1740-1744). For the case $n=1, p\in(1,2)$, to justify integration by parts and Poincar\'e inequality, we observe that for $t>0: l(t)\in W^2_p(\Omega)\subset W^1_2(\Omega)$. 
\end{proof} 

\subsection{$L_{p}$ estimate}
In this subsection we will prove that for $t\geq0$
\eq{
\n{z_1(t)}_p+\n{z_2(t)}_p\leq Ce^{-(\chi/2)t},\label{estimate:Lp}
}
the parameter $p$ being defined in $\textbf{A1}$.
\\
\newline
Notice that if $p\in (1,2]$ (which can only happen if $n=1$), the inequality \eqref{estimate:Lp} follows from \eqref{KLW}. 
\\
\newline
Otherwise we have $p>(2\vee n)$. We choose an increasing sequence $(p_i)_{i=1}^m$ such that 
\begin{align*}
p_{1}=2, p_{m}=p\\
n/2(1/p_i-1/p_{i+1})<1
\end{align*}
(notice that for $n\in\{1,2,3,4\}$ one can take $m=2$). 
Inductively we will prove that 
\eq
{
\n{z_1(t)}_{p_i}+\n{z_2(t)}_{p_i}\leq Ce^{-(\chi/2)t}, \ 1\leq i\leq m\label{estimate:Lpi}.
} 
For $i=1$ \eqref{estimate:Lpi} follows from \eqref{KLW}.
Assume that \eqref{estimate:Lpi} is true for some $1\leq i\leq m-1$. Then
\eq{
\n{H^1(z(t))}_{p_i}\leq\n{z_1}_{p_i}\n{1-z_2+s_{\infty}}_{\infty}+\n{z_2}_{p_i}\n{\delta+\epsilon+l_{\infty}}_{\infty}\leq Ce^{-(\chi/2)t}\label{H1Lp}.
}

Using \eqref{Duhamel1}, \eqref{estimate:2}, \eqref{H1Lp} and $\chi/2< D\lambda_1$ we obtain
\begin{align*}
\n{z_1(t)}_{p_{i+1}}&\leq\n{e^{tD\Delta}z_{10}}_{p_{i+1}}+\int_{0}^{t}\n{e^{sD\Delta}H^{1}(z(t-s))}_{p_{i+1}}ds\\
&\leq Ce^{-D\lambda_1t}+C\int_{0}^{t}(Ds\wedge1)^{-n/2(1/p_{i}-1/p_{i+1})}e^{-D\lambda_1s}\n{H^{1}(z(t-s))}_{p_{i}}ds\\
&\leq Ce^{-D\lambda_1t}+C\int_{0}^{t}(Ds\wedge1)^{-n/2(1/p_i-1/p_{i+1})}e^{-D\lambda_1s}e^{-(\chi/2)(t-s)}ds\\
&\leq Ce^{-D\lambda_1t}+Ce^{-(\chi/2)t}\int_{0}^{t}(Ds\wedge1)^{-n/2(1/p_i-1/p_{i+1})}e^{-(D\lambda_1-\chi/2)s}ds\\
&\leq Ce^{-(\chi/2)t}.
\end{align*}

To show that for $t>0 \ $ 
$\n{z_2(t)}_{p_{i+1}}\leq Ce^{-(\chi/2)t}$,
we obtain from Theorem \ref{global} that for each fixed $x\in\overline{\Omega}$ the function $z_{2}\in\mathcal{C}^1([0,\infty);\mathbb{R})$ satisfies the ODE
\begin{align*}
\dot{z}_2+(\delta+\epsilon+l_{\infty}+z_1)z_2&=(1-s_{\infty})z_1,
\end{align*}
hence
\eq{
z_{2}(t)=A(t)z_{20}+(1-s_{\infty})\int_{0}^{t}A(\tau)z_{1}(t-\tau)d\tau\label{ODEz2},
}
where 
\eq{
A(t)=\exp{\Big(-\int_{0}^{t}(\delta+\epsilon+l_{\infty}+z_{1}(\tau))d\tau\Big)}.
\label{At}
}
From $l_{\infty}+z_{1}=l\geq0$ we get $\n{A(t)}_{\infty}\leq e^{-(\delta+\epsilon)t}$. Using $\chi/2<\delta+\epsilon$ we obtain
\begin{align*}
\n{z_2(t)}_{p_{i+1}}&\leq\n{A(t)}_{\infty}\n{z_{20}}_{p_{i+1}}+\n{1-s_{\infty}}_{\infty}\int_0^t \n{A(\tau)}_{\infty}\n{z_{1}(t-\tau)}_{p_{i+1}}d\tau\\
&\leq Ce^{-(\delta+\epsilon)t}+Ce^{-(\chi/2)t}\int_{0}^{t}e^{-(\delta+\epsilon-\chi/2)\tau}d\tau\leq Ce^{-(\chi/2)t},
\end{align*}
thus finishing the proof of \eqref{estimate:Lpi}, whence that of \eqref{estimate:Lp}.
\\
\newline

In the next two sections we use the smoothing properties of $e^{t\Delta}$ to extend convergence to the first and second derivatives.

\subsection{$W^{1}_{p}$ estimate}
Using \eqref{Duhamel1}, \eqref{estimate:1}, \eqref{estimate:Lp} and $\chi/2<D\lambda_1$ we obtain
\begin{align*}
\n{z_1(t)}_{1,p}&\leq\n{e^{tD\Delta}z_{10}}_{1,p}
+\int_{0}^{t}\n{e^{sD\Delta}H^{1}(z(t-s))}_{1,p}ds\\
&\leq Ce^{-D\lambda_1t}+C\int_{0}^{t}(Ds\wedge 1)^{-1/2}e^{-\lambda_1Ds}\n{H^{1}(z(t-s))}_{p}ds\\
&\leq Ce^{-D\lambda_1t}+C\int_{0}^{t}(Ds\wedge 1)^{-1/2}e^{-\lambda_1Ds}e^{-(\chi/2)(t-s)}ds\\
&\leq Ce^{-D\lambda_1t}+Ce^{-(\chi/2)t}\int_{0}^{t}(Ds\wedge 1)^{-1/2}e^{-(D\lambda_1-\chi/2)s}ds\\
&\leq Ce^{-(\chi/2)t}.
\end{align*}

Using the above estimate for $z_1$ we obtain that $A(t)$ given by \eqref{At} satisfies
\begin{align*}
\n{A(t)}_p&\leq C\n{A(t)}_{\infty}\leq Ce^{-(\delta+\epsilon)t}\\
\n{\nabla A(t)}_p&=\n{-A(t)\int_{0}^{t}(\nabla l_{\infty}+\nabla z_{1}(\tau))d\tau}_p\leq\n{A(t)}_{\infty}\int_{0}^{t}(\n{\nabla l_{\infty}}_p+\n{\nabla z_{1}(\tau)}_p)d\tau\\
&\leq Ce^{-(\delta+\epsilon)t}\int_0^{t}(1+e^{-(\chi/2)\tau})d\tau\leq Cte^{-(\delta+\epsilon)t}.
\end{align*}
Thus using \eqref{ODEz2} we have
\begin{align*}
\n{z_2(t)}_{1,p}&\leq \n{A(t)}_{1,p}\n{z_{20}}_{1,p}+C\n{1-s_{\infty}}_{1,p}\int_{0}^{t}\n{A(\tau)}_{1,p}\n{z_{1}(t-\tau)}_{1,p}d\tau\\
&\leq C(t+1)e^{-(\delta+\epsilon)t}+C\int_{0}^{t}(\tau+1)e^{-(\delta+\epsilon)\tau}e^{-(\chi/2)(t-\tau)}d\tau\\
&\leq C(t+1)e^{-(\delta+\epsilon)t}+Ce^{-(\chi/2)t}\int_{0}^{t}(\tau+1)e^{-(\delta+\epsilon-\chi/2)\tau}d\tau\\
&\leq Ce^{-(\chi/2)t}
\end{align*}
which finishes the proof of \eqref{asym1}. 

\subsection{$W^2_p$ estimate for $z_1$}
Using \eqref{Duhamel1}, \eqref{estimate:1}, \eqref{asym1} and $\chi/2<D\lambda_1$ we obtain
\begin{align*}
\n{z_1(t)}_{2,p}&\leq\n{e^{tD\Delta}z_{10}}_{2,p}
+\int_{0}^{t}\n{e^{\tau D\Delta}H^{1}(z(t-\tau))}_{2,p}d\tau\\
&\leq C(Dt\wedge1)^{-1/2}e^{-D\lambda_1t}+C\int_{0}^{t}(D\tau\wedge 1)^{-1/2}e^{-\lambda_1D\tau}e^{-(\chi/2)(t-\tau)}ds\\
&\leq C(t\wedge1)^{-1/2}e^{-D\lambda_1t}+Ce^{-(\chi/2)t}\int_{0}^{t}(\tau\wedge 1)^{-1/2}e^{-(D\lambda_1-\chi/2)\tau}d\tau\\
&\leq C\max\{1/\sqrt{t},1\}e^{-(\chi/2)t},
\end{align*}

which finishes the proof of \eqref{asym2}. 

\section{Acknowledgement}
The author would like to express his gratitude towards his PhD supervisors Philippe Lauren\c{c}ot and Dariusz Wrzosek for their constant encouragement and countless helpful remarks and towards Christoph Walker for discussions on interpolation techniques and multiplication in Sobolev spaces. \\
The author was supported by the International Ph.D. Projects Programme of Foundation for Polish Science operated within the Innovative Economy Operational Programme 2007-2013 funded by European
Regional Development Fund (Ph.D. Programme: Mathematical Methods in Natural Sciences). \\
Part of this research was carried out during author's visit to the Institut de Math\'ematiques de Toulouse.

\end{document}